\documentclass{basm}

\author{R.R. Kamalian}

\title{Examples of cyclically-interval non-colorable bipartite graphs}

\RunningHead{R.R. Kamalian}%
            {Examples of cyclically-interval non-colorable bipartite graphs}

\Keywords{cyclically-interval edge coloring, bipartite graph}

\MSC{05C15}

\Abstract{For an undirected, simple, finite, connected graph $G$, we
denote by $V(G)$ and $E(G)$ the sets of its vertices and edges,
respectively. A function $\varphi:E(G)\rightarrow\{1,2,\ldots,t\}$
is called a proper edge $t$-coloring of a graph $G$ if adjacent
edges are colored differently and each of $t$ colors is used. An
arbitrary nonempty subset of consecutive integers is called an
interval. If $\varphi$ is a proper edge $t$-coloring of a graph $G$
and $x\in V(G)$, then $S_G(x,\varphi)$ denotes the set of colors of
edges of $G$ which are incident with $x$. A proper edge $t$-coloring
$\varphi$ of a graph $G$ is called a cyclically-interval
$t$-coloring if for any $x\in V(G)$ at least one of the following
two conditions holds: a) $S_G(x,\varphi)$ is an interval, b)
$\{1,2,\ldots,t\}\setminus S_G(x,\varphi)$ is an interval. For any
$t\in \mathbb{N}$, let $\mathfrak{M}_t$ be the set of graphs for
which there exists a cyclically-interval $t$-coloring, and let
$$\mathfrak{M}\equiv\bigcup_{t\geq1}\mathfrak{M}_t.$$ Examples of
bipartite graphs that do not belong to the class $\mathfrak{M}$ are
constructed.}

\CopyRight { R.R. Kamalian 2013}

\Address{{\sc R.R. Kamalian}\\
{Institute for Informatics and Automation Problems National Academy
of Sciences of RA, 0014 Yerevan, Republic of Armenia}\\
E-mail: \emph{rrkamalian@yahoo.com}}
\medskip

\begin{document}

\maketitle

\section{Introduction}
We consider undirected, simple, finite, and connected graphs. For a
graph $G$ we denote by $V(G)$ and $E(G)$ the sets of its vertices
and edges, respectively. For a graph $G$, we denote by $\Delta (G)$
and $\chi ^{\prime }(G)$ the maximum degree of a vertex of $G$ and
the chromatic index of $G$ \cite{Vizing2}, respectively. The terms
and concepts which are not defined can be found in \cite{West1}.

For an arbitrary finite set $A$, we denote by $|A|$ the number of
elements of $A$. The set of positive integers is denoted by
$\mathbb{N}$. An arbitrary nonempty subset of consecutive integers
is called an interval. An interval with the minimum element $p$ and
the maximum element $q$ is denoted by $[p,q]$.

For any $t\in \mathbb{N}$ and arbitrary integers $i_{1},i_{2}$
satisfying the conditions $i_{1}\in[1,t]$, $i_{2}\in[1,t]$, we
define \cite{Shved1_11, Moldova} the sets
$intcyc_{1}[(i_{1},i_{2}),t],$ $intcyc_{1}((i_{1},i_{2}),t),$
$intcyc_{2}((i_{1},i_{2}),t),$ $intcyc_{2}[(i_{1},i_{2}),t]$ as
follows:
$$intcyc_{1}[(i_{1},i_{2}),t]\equiv [\min
\{i_{1},i_{2}\},\max\{i_{1},i_{2}\}],$$
$$intcyc_{1}((i_{1},i_{2}),t)\equiv
intcyc_{1}[(i_{1},i_{2}),t]\backslash (\{i_{1}\}\cup \{i_{2}\}),$$
$$intcyc_{2}((i_{1},i_{2}),t)\equiv [1,t]\backslash
intcyc_{1}[(i_{1},i_{2}),t],$$
$$intcyc_{2}[(i_{1},i_{2}),t]\equiv [1,t]\backslash
intcyc_{1}((i_{1},i_{2}),t).$$

If $t\in \mathbb{N}$ and $Q$ is a non-empty subset of the set
$\mathbb{N}$, then $Q$ is called a $t$-cyclic interval if there
exist integers $i_{1},i_{2},j_{0}$ satisfying the conditions
$i_{1}\in[1,t]$, $i_{2}\in[1,t]$, $ j_{0}\in \{1,2\}$,
$Q=intcyc_{j_{0}}[(i_{1},i_{2}),t]$.

A function $\varphi:E(G)\rightarrow [1,t]$ is called a proper edge
$t$-coloring of a graph $G$ if adjacent edges are colored
differently and each of $t$ colors is used.

For a graph $G$ and a positive integer $t$, where $\chi'(G)\leq
t\leq|E(G)|$, we denote by $\alpha(G,t)$ the set of all proper edge
$t$-colorings of $G$. Let us set
$$
\alpha(G)\equiv \bigcup_{t=\chi'(G)}^{|E(G)|}\alpha(G,t).
$$

If $G$ is a graph, $\varphi\in\alpha(G)$, and $x\in V(G)$, then the
set $\{\varphi(e)/ e\in E(G), e \textrm{ is incident with } x\}$ is
denoted by $S_G(x,\varphi)$.

A proper edge $t$-coloring $\varphi$ of a graph $G$ is called a
cyclically-interval $t$-coloring of $G$, if for any $x\in V(G)$ at
least one of the following two conditions holds: a) $S_G(x,\varphi)$
is an interval, b) $[1,t]\setminus S_G(x,\varphi)$ is an interval.

For any $t\in \mathbb{N}$, we denote by $\mathfrak{M}_t$ the set of
graphs for which there exists a cyclically-interval $t$-coloring.
Let
$$\mathfrak{M}\equiv\bigcup_{t\geq1}\mathfrak{M}_t.$$

For an arbitrary tree $D$, it was shown in \cite{Shved1_11, Moldova}
that $D\in\mathfrak{M}$, and, moreover, all possible values of $t$
were found for which $D\in\mathfrak{M}_t$. For an arbitrary simple
cycle $C$, it was shown in \cite{Csit10, China} that
$C\in\mathfrak{M}$, and, moreover, all possible values of $t$ were
found for which $C\in\mathfrak{M}_t$. Some interesting results on
this and related topics were obtained in \cite{Alt, Barth8, Daus9,
DeWerra7, DeWerra6, Kubale, Kubale_Nadolski, Nadolski,
Asratian_Diss, Jensen_Toft, Diss3}.

In this paper, the examples of bipartite graphs that do not belong
to the class $\mathfrak{M}$ are constructed.

For any integer $m\geq2$, set:
$$
V_{0,m}\equiv\{x_0\},\qquad V_{1,m}\equiv\{x_{i,j}/\;1\leq i<j\leq
m\},
$$
$$
V_{2,m}\equiv\{y_{p,q}/\;1\leq p\leq m,1\leq q\leq m\},
$$
$$
E'_m\equiv\{(x_0,y_{p,q})/\;1\leq p\leq m,1\leq q\leq m\}.
$$

For any integers $i,j,m$ satisfying the inequalities $m\geq2$,
$1\leq i<j\leq m$, set:
$$
E''_{i,j,m}\equiv\{(x_{i,j},y_{i,q})/\;1\leq q\leq
m\}\cup\{(x_{i,j},y_{j,q})/\;1\leq q\leq m\}.
$$

For any integer $m\geq2$, let us define a graph $G(m)$ by the
following way:
$$
G(m)\equiv\Bigg(\bigcup_{k=0}^{2}V_{k,m},E'_m\cup\Bigg(\bigcup_{1\leq
i<j\leq m}E''_{i,j,m}\Bigg)\Bigg).
$$

It is not difficult to see that for any integer $m\geq2$, $G(m)$ is
a bipartite graph with $\Delta(G(m))=\chi'(G(m))=m^2$,
$|V(G(m))|=\frac{3m^2-m}{2}+1$, $|E(G(m))|=m^3$.

\begin{thm}
For any integer $m\geq8$, $G(m)\not\in\mathfrak{M}$.
\end{thm}

\begin{proof}
Assume the contrary. It means that there exist integers
$m_0,t_0,k_0$, satisfying the conditions $m_0\geq8$, $m_0^2\leq
t_0\leq m_0^3$, $t_0=m_0^2+k_0$, $0\leq k_0\leq m_0^3-m_0^2$,
$G(m_0)\in\mathfrak{M}_{t_0}$.

Let $\varphi_0$ be a cyclically-interval $t_0$-coloring of the graph
$G(m_0)$. Without loss of generality, we can suppose that
$S_{G(m_0)}(x_0,\varphi_0)=[1,m_0^2]$. Let us consider the edges
$e'$ and $e''$ of the graph $G(m_0)$, which are incident with the
vertex $x_0$ and satisfy the equalities $\varphi_0(e')=1$,
$\varphi_0(e'')=\lfloor\frac{m_0^2}{2}\rfloor$.

Suppose that $e'=(x_0,y')$, $e''=(x_0,y'')$. Clearly, there exists a
vertex $\widetilde{x}\in V_{1,m_0}$ in the graph $G(m_0)$ which is
adjacent to the vertices $y'$ and $y''$. It is not difficult to see
that $S_{G(m_0)}(y',\varphi_0)\cup
S_{G(m_0)}(\widetilde{x},\varphi_0)\cup S_{G(m_0)}(y'',\varphi_0)$
is a $t_0$-cyclic interval.

Clearly, the inequalities $m_0^2+k_0-4m_0+4>4m_0-2$ and
$4m_0-1\leq\lfloor\frac{m_0^2}{2}\rfloor\leq m_0^2+k_0-4m_0+3$ are
true. Consequently, $\lfloor\frac{m_0^2}{2}\rfloor\in
intcyc_1((4m_0-2,m_0^2+k_0-4m_0+4),m_0^2+k_0)$. But it is
incompatible with the evident relations
$\lfloor\frac{m_0^2}{2}\rfloor\in S_{G(m_0)}(y'',\varphi_0)$ and
$S_{G(m_0)}(y',\varphi_0)\cup
S_{G(m_0)}(\widetilde{x},\varphi_0)\cup
S_{G(m_0)}(y'',\varphi_0)\subseteq
intcyc_2[(4m_0-2,m_0^2+k_0-4m_0+4),m_0^2+k_0]$. Contradiction.

\end{proof}

The author thanks P.A. Petrosyan for his attention to this work.

\end{document}